%% file: main.tex
\documentclass[]{amsart}
\usepackage{amsthm,amssymb,amsthm}
\usepackage{verbatim}
\usepackage{hyperref}
\usepackage{mathpazo}
\theoremstyle{plain}
\newtheorem{theorem}{Theorem}[section]
\newtheorem{lemma}{Lemma}[section]

\newtheorem{corollary}[theorem]{Corollary}
\theoremstyle{definition}
\newtheorem{definition}{Definition}
\theoremstyle{remark}
\newtheorem*{remark}{Remark}
\usepackage{subfigure}
\usepackage{boxedminipage}

\usepackage{ifpdf}

\newcommand{\rr}{\mathbb{R}}
\newcommand{\acc}{\text{ACC}}

\ifpdf
\usepackage[pdftex]{graphicx}
\else
\usepackage{graphicx}
\fi
\title{Isoperimetric inequalities for wave fronts and a generalization of Menzin's conjecture for bicycle monodromy on surfaces of constant curvature}

\author{Sean Howe}
\address{Sean Howe \\ University of Arizona Math Dept.} 
\email{spkh@email.arizona.edu}
\thanks{This paper comes out of the Penn. State University 2008 Summer REU. Thanks to Sergei Tabachnikov for his guidance, Suchandan Pal for his contributions to our work, and to the NSF and PSU for their financial support.}

\author{Matthew Pancia}
\address{Matthew Pancia \\ University of Texas at Austin Math Dept.}
\email{mpancia@math.utexas.edu}

\author{Valentin Zakharevich}
\address{Valentin Zakharevich \\ Polytechnic Institute of NYU Math Dept.}
\email{vzakha02@students.poly.edu}
\date{\today}

\begin{document}
\pagestyle{plain}
\ifpdf
\DeclareGraphicsExtensions{.pdf, .jpg, .tif}
\else
\DeclareGraphicsExtensions{.eps, .jpg}
\fi

\begin{abstract}
	The classical isoperimetric inequality relates the lengths of curves to the areas that they bound. More specifically, we have that for a smooth, simple closed curve of length $L$ bounding area $A$ on a surface of constant curvature $c$, 
	\begin{equation*}
	L^2 \geq 4 \pi A - cA^2
	\end{equation*}
	with equality holding only if the curve is a geodesic circle. We prove generalizations of the isoperimetric inequality for both spherical and hyperbolic wave fronts (i.e. piecewise smooth curves which may have cusps).
	We then discuss ``bicycle curves'' using the generalized isoperimetric inequalities. The euclidean model of a bicycle is a unit segment $AB$ that can move so that it remains tangent to the trajectory of point $A$ (the rear wheel is fixed on the bicycle frame), as discussed in \cite{finncycle}, \cite{tabachold}, and \cite{tabachnew}. We extend this definition to a general Riemannian manifold, and concern ourselves in particular with bicycle curves in the hyperbolic plane $H^2$ and on the sphere $S^2$. We prove results along the lines of those in \cite{tabachnew} and resolve both spherical and hyperbolic versions of Menzin's conjecture, which relates the area bounded by a curve to its associated monodromy map. 
\end{abstract}

\maketitle

\section{Introduction}
	\subsection{Preliminaries}
	We will use the standard model of the sphere, $S^2$, as an embedded surface in $\mathbb{R}^3$, consisting of all position vectors with euclidean magnitude 1. 
	There are several models of the hyperbolic plane, $H^2$, but we will find it convenient to use the hyperboloid model \footnote{For more information about hyperbolic geometry, see \cite{hypergeom} or \cite{hypergeom2}.}.
	In the hyperboloid model, the hyperbolic plane is realized as the set of points in $\mathbb{R}^3$ whose position vectors have $z$ coordinate $> 0$ and norm -1 with respect to the quadratic form
	\[ds^2 = dx^2 + dy^2 - dz^2. \] 
	This corresponds to the positive sheet of the hyperboloid of two sheets embedded in $\rr^3$. Restricted to the tangent space of the hyperboloid, the above quadratic form is non-negative, and therefore defines a genuine Riemannian metric on $H^2$.
	
	The primary object we will be working with in this paper is called
	a wave front. 
	\begin{definition} Let $\gamma$ be a curve. We say that $\gamma$ is a \emph{wave front} if it is piecewise smooth and its only singularities are (semicubical) cusps. \end{definition} 
	
	\subsection{The Isoperimetric Inequality}
	\input{inequality}

\section{On The Sphere}
	\subsection{Spherical Areas}
	\input{areas}
	\subsection{A Spherical Isoperimetric Inequality}
	Using the above, we are in a position to prove a spherical version of the classical isoperimetric inequality for wave fronts. 
	\begin{theorem}\label{thm:wcurve}
	Let $\gamma$ be a wave front without inflection points and having an even number of double-tangent points.
	Then 
	\begin{equation}\label{wcurve}
		\acc(\gamma)^{2} + L(\gamma)^{2} \geq 4 \pi^{2}.
	\end{equation}
	\end{theorem}

	\begin{proof}
		Consider the dual to $\gamma$. We have that $\gamma^*$ is smooth because $\gamma$ has no inflection points and is also homotopic to a circle traversed once because $\gamma$ has an even number of double-tangent points. By \cite{weiner} we then have that
		\[ 
		\acc(\gamma^*)^{2} + L(\gamma^*)^{2} \geq 4 \pi^{2}. 
		\]
		Using \eqref{arealength}, we obtain the desired result. 
	\end{proof}
\section{In the Hyperbolic Plane}
	\subsection{A Hyperbolic Isoperimetric Inequality}
	\input{hyperbolic}
\section{An Application to Bicycle Curves}
We now discuss the subject of bicycle curves, where isoperimetric inequalities for wave fronts will prove to be a useful tool. 

We first define our model of bicycle motion. The bicycle frame is represented by a geodesic line segment $AB$ of constant length $l$, while the back and front wheels are represented by $A$ and $B$ respectively. The path of the front wheel is restricted so that $AB$ is always tangent to the velocity of the back wheel, $A$. This is illustrated in Figure \ref{fig:notation2}. 
	\begin{figure}[htbp]
		\centering
			\includegraphics[scale=1]{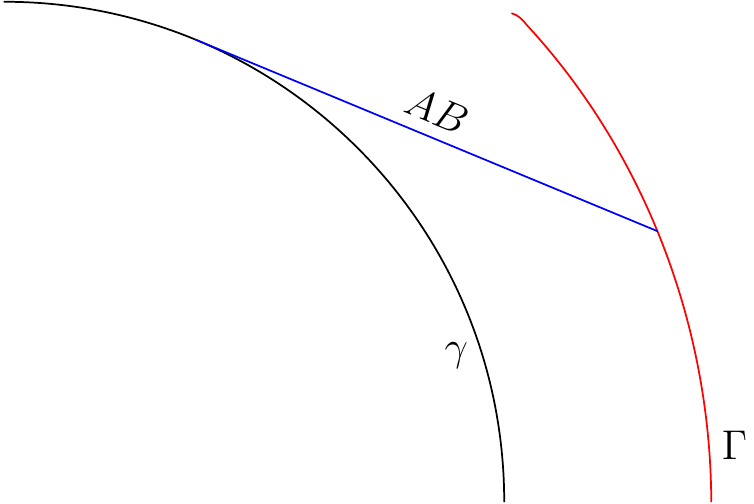}
		\caption{The bicycle model.}
		\label{fig:notation2}
	\end{figure}
	An important construction that we will consider is the monodromy map associated to a particular front wheel curve. 
	\begin{definition}
		Let $\Gamma$ be a curve representing the motion of the front wheel. Up to choice of relative initial position of the back wheel, the path of the back wheel is determined by $\Gamma$. The \emph{monodromy map} associated to $\Gamma$ or $M(\Gamma): S^1 \rightarrow S^1$ is the function that sends each choice of starting position for the back wheel to its terminal position. 
	\end{definition}

Finally, let's establish the notation that will be used for the remainder of the paper. The space that each result applies to ($S^2$ or $H^2$) will be noted in parentheses. By $AB$ denote the segment representing the bicycle frame from back wheel to front. $\Gamma$ will denote the path of the front wheel and $\gamma$ the path of the back wheel. Let $s$ be the arc length parameter for $\Gamma$ and $t$ the arc length parameter for $\gamma$, $\kappa$ and $k$ the respective curvatures. Let $\alpha$ be the angle between $\Gamma'$ and $BA$, which can be parameterized by either $s$ or $t$. Finally, $L(\gamma)$ will denote the algebraic length of the curve $\gamma$. 
	\subsection{General Results}
		In this section we develop some general results about bicycle curves in $H^2$ and $S^2$ that will allow us to describe how they evolve. These results will primarily concern themselves with the relationships between properties of the front and back wheel curves, which will be essential to establishing our later results about the monodromy map and our reformulation of Menzin's conjecture. 
		\input{sphericaldiffeq}
		The above results on the sphere have parallels in the hyperbolic plane, which we prove below. Note that they are identical to the spherical results up to the replacement of trigonometric functions with hyperbolic trigonometric functions.
		\input{hyperbolicdiffeq}
	\subsection{The Monodromy Map}
		\input{monodromy}
	\subsection{A Spherical Menzin's Conjecture} 
		\input{menzin}
	\subsection{A Hyperbolic Menzin's Conjecture}
		\input{hyperbolicmenzin}
\section{Remarks and Conclusion}
Our initial goal in this paper was to generalize to surfaces of constant
curvature the results in \cite{tabachnew} on bicycle monodromy
in the Euclidean plane. Largely we have succeeded, providing analogs
in both $S^{2}$ and $H^{2}$ to all of their major results. In addition,
our method of attack led to the development of a useful tool - the
isoperimetric inequality for wavefronts. However, there is one aspect
of Menzin's conjecture in $H^{2}$ that we have not been able to satisfactorily
address. In both $E^{2}$ and $S^{2}$ we require convexity for the
front wheel curve, but in $H^{2}$ our version of the conjecture requires
the front wheel curve to be horocyclically convex. The question is
then whether this requirement is necessary or just an artifact of
our proof technique. Our guess is the former, but we lack a proof
of the fact or an example of a convex but not horocyclically convex
front wheel path in the hyperbolic plane bounding the correct amount
of area but without hyperbolic monodromy. This question will need
to be resolved before we can really close the book on Menzin's conjecture
in $H^{2}.$
\bibliographystyle{amsplain}
\bibliography{bib}
\end{document}

%% file: inequality.tex
As our paper deals with a generalization of the classical isoperimetric inequality to wave fronts, let us first discuss the history of the problem. The classical isoperimetric inequality relates the length of a plane curve to the area that it bounds. More precisely we have that for a smooth, simple closed curve of length $L$ bounding area $A$ on a surface of constant curvature $c$, 
\begin{equation*}
L^2 \geq 4 \pi A - cA^2
\end{equation*}
with equality holding only if the curve is a geodesic circle.

T.F. Banchoff and W.F. Pohl \cite{banchoff1971gii} generalized the isoperimetric inequality to non-simple curves 
in the euclidean plane. In the case of a non-simple curve one must redefine the notion of area. $A$ is replaced by the sum of the areas into which the curve divides the plane, weighted by the square of the winding number:
\begin{equation*}
L^2 \geq 4 \pi \int\limits_{\mathbb{R}^{2}} w_{f}^{2}(p)dA
\end{equation*}
where $f$ is a smooth immersion and $w_{f}(p)$ is a winding number of $f$ with respect to the point $p$. 

J.L. Weiner generalized this result to smooth immersions of a circle into the 2-sphere, in \cite{weiner} (which we use to prove some later results). Using the Gauss-Bonnet Theorem, the classical isoperimetric inequality can be written as 
\begin{equation}\label{Weiner ineq}
L^2 + K^2 \geq 4 \pi ^2 
\end{equation}
where $K = \int k ds$, with $k$ being the geodesic curvature and $ds$ the element of arc length. This inequality makes sense for immersions as well as embeddings. In \cite{weiner}, Weiner shows that \eqref{Weiner ineq} holds for smooth immersions regularly homotopic to a circle traversed once.

E. Teufel made a similar generalization for smooth immersions of a circle into the hyperbolic plane, in \cite{MR1129528}. Teufel showed that
\begin{equation*}
L^2 \geq 4 \pi \int\limits_{H^2}w_f^2(p)dH_p + \int\limits_{H^2\times H^2}w_f(p)w_f(q)dH_p \wedge dH_q
\end{equation*}
Here $f$ is a smooth immersion of a circle into the hyperbolic plane $H^2$, $w_f(p)$ is the winding number of $f$ with respect to the 
point $p$, and $dH_p$ is the area element of $H$ at the point $p$. In all three geometries the equality holds only if the curve traverses a geodesic circle a number of times in the same direction.

One of the goals of this paper will be to generalize the isoperimetric inequality to wave fronts. In \cite{hedgehogs}, Martinez-Maure has obtained similar inequalities for euclidean wave fronts, but we concern ourselves mainly with the hyperbolic plane and the sphere.

%% file: areas.tex
Because the euclidean and hyperbolic planes share the property that all closed forms are exact, we have an unambiguous way of defining area. In the case of the plane, we can say that the area bounded by an oriented, simple curve $\gamma$ is the line integral 
\[ \int_\gamma x \; dy .\] Note that this integral can be rewritten as some integral with respect to the arclength element of the curve, $ds$. By Stokes' theorem, this is the same as integrating $dx \wedge dy$ over the region which has $\gamma$ as its boundary (with the induced orientation). We can extend this definition to wave fronts by signing the arclength element so that it changes when passing through cusps.

Due to the topology of the sphere, however, we do not have that its area form is the exterior derivative of any $1-$form. A way to get around this is to define a particular 2-chain which has a given curve as its boundary, and declare the area bounded by the curve to be the integral of the area form over this chain. This is the approach taken in \cite{arnoldsphere}, but we will concern ourselves with a more naive notion of area. Our definition of area will be restricted to a small class of curves, and is motivated by the Gauss-Bonnet Theorem.

\begin{definition}
	For a convex, simple, smooth spherical curve $\Gamma$, we define the \emph{area bounded by $\Gamma$} or $A(\Gamma)$ as 
	\[ A(\Gamma) = 2\pi - \int_\Gamma \kappa \]
	where $\kappa$ is the geodesic curvature of $\Gamma$ and the curve is oriented so that $\int_\Gamma \kappa$ is positive. In this case, we say that $\Gamma$ is \emph{properly oriented}.
\end{definition}

We also need a notion of algebraic length, defined for oriented, co-oriented wave fronts on any surface. Note, if we have such a wave front, the orientation and co-orientation together induce a sign on the arclength element of the curve. This is done by signing the arclength element with the sign of the frame formed by the co-orientation and orientation.  
\begin{definition}\label{def:}
	Let $\gamma$ be an oriented, co-oriented wave front, we define the \emph{algebraic length of $\gamma$}, $L(\gamma)$, to be 
	\[ \int_\gamma ds\]
	where $ds$ is the element of arclength, signed as described above. 
\end{definition}

The approach taken in \cite{arnoldsphere} by Arnold is significantly more complicated, but it allows us to make a reasonable definition of the area bounded by a much wider variety of curves, including those with cusps. The construction of this area (which we will denote by $ACC$ to avoid ambiguity) is analogous to using the winding number to define area in the plane, and is as follows (more specific details can be found in \cite{arnoldsphere}, we present only the basics). 

	For a closed, oriented, co-oriented spherical wave front $\gamma$, we have that $\gamma$ divides the sphere into distinct regions. For each region, we pick a point (not on $\gamma$) in this region and do stereographic projection using this point as the point at infinity. This turns $\gamma$ into a planar curve, whose normal vector has a winding number $i$ (the number of turns it makes around the unit circle after normalization, which is dependent only on the choice of region). To each region, we assign the coefficient $-\frac{i}{2}$. 
\begin{definition}
	The \emph{characteristic 2-chain} of $\gamma$ is the formal sum of the regions that $\gamma$ divides the sphere into with coefficients computed in the manner described above. Note that $\gamma$ is the boundary of this 2-chain, by a result in \cite{arnoldsphere}.
\end{definition}
Using this construction, we are now in a position to define a type of area on the sphere. 
\begin{definition}
	Let $c$ be the characteristic 2-chain of a wave front $\gamma$. Then the \emph{area of the characteristic 2-chain} of $\gamma$ or $\acc(\gamma)$ as the integral of the spherical area form over this 2-chain. That is, if $\omega$ is the spherical area form,
	\[ \acc(\gamma) = \int_c \omega. \]
\end{definition}
With this definition of area, we have a version of the Gauss-Bonnet Theorem (taken from \cite{arnoldsphere}) which will be useful for us later.
\begin{theorem}
	For an oriented, co-oriented front $\gamma$ with geodesic curvature $k$, we have 
	\begin{equation}
		\acc(\gamma) = \int_\gamma k(s) \; ds, 
	\end{equation}
	where $ds$ is the signed element of arclength, determined by the co-orientation of $\gamma$.
\end{theorem}

\begin{definition}\label{def:dual}
	Given an oriented, co-oriented front $\gamma$ on the sphere, we can obtain the \emph{dual} of $\gamma$, denoted $\gamma^*$, by moving every point of $\gamma$ a distance of $\pi \over 2$ along the great circle in the direction of the co-orientation. 
\end{definition}
We have that taking the dual of a curve turns double tangent points into self-intersections and turns cusps into inflections, as seen in Figure \ref{dualfig}. As well, from \cite{arnoldsphere} we have the following relations between the area and length of a curve and its dual. 
\begin{figure}[htbp]
	\centering
		\includegraphics[scale=.7]{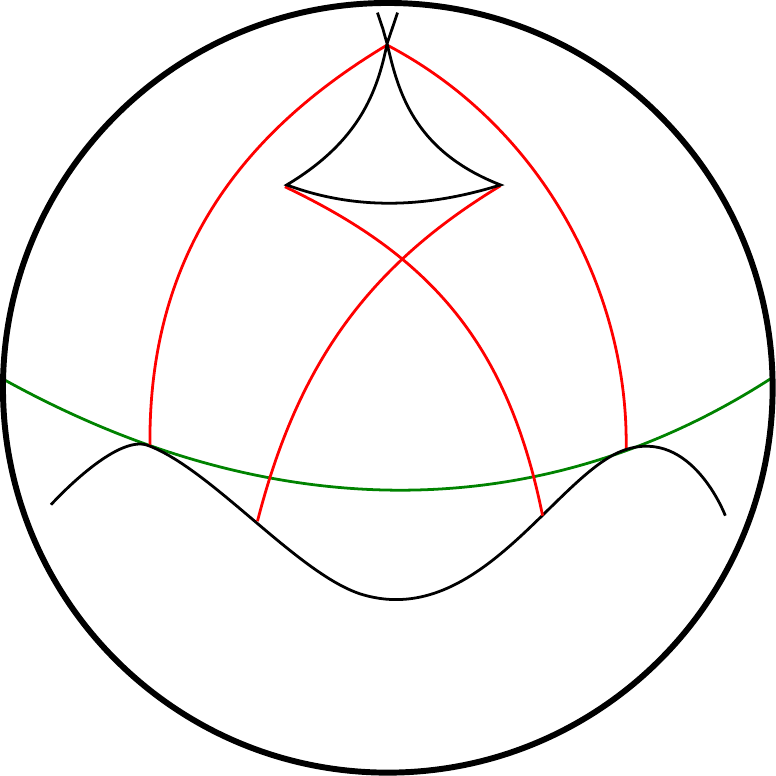}
	\caption{A spherical curve and its dual.}
	\label{dualfig}
\end{figure}
\begin{theorem}\label{thm:arealength}
	Let $\gamma$ be a closed, oriented, co-oriented wave front on the sphere. Then we have that 
	\begin{equation}\label{arealength}
		\acc(\gamma^*) = L(\gamma) \qquad \text{and} \qquad \acc(\gamma) = -L(\gamma^*). 
	\end{equation}
\end{theorem}

%% file: hyperbolic.tex
	Now we seek to prove a similar isoperimetric inequality for hyperbolic wave fronts, using an approach similar to those taken by Anisov in \cite{arealength} and \cite{intwavefront}. Analogously to the approach we took on the sphere, we are going to introduce a new function $C(\gamma) = \displaystyle\int_{\gamma}k_{\gamma}$ where $k_{\gamma}$ is the geodesic curvature of $\gamma$. We are going to let $A(\gamma)$ be defined for simple curves as the regular area of $\gamma$. From the Gauss-Bonnet theorem, we have that $C(\gamma) = A(\gamma) + 2\pi$.
	
	\begin{theorem}
	Let $\gamma$ be a closed, oriented, co-oriented and horocyclically convex wave front (i.e. its geodesic curvature has magnitude that is everywhere $\geq 1$) with turning number 1 in $H^2$. Let $C(\gamma) = \displaystyle\int_{\gamma}k_{\gamma}$, where the arc 
	length changes sign at every cusp, and let $L(\gamma)$ be its algebraic length. Then we have 
	\begin{equation*}
	L(\gamma)^2 + 4\pi^2 - C(\gamma)^2 \geq 0,
	\end{equation*}
	with equality iff $\gamma$ is a geodesic circle. 
	\end{theorem}
	
	\begin{proof}
	Consider the family $\gamma_{t}$ of equidistant fronts of $\gamma$. Let $\overline{C}(t)$ and $\overline{L}(t)$ be $C(\gamma_{t})$ and 
	$L(\gamma_{t})$ respectively. We are going to show that $\overline{C'}(t) = \overline{L}(t)$ and $\overline{L'}(t) = \overline{C}(t)$ (where $'$ denotes differentiation with respect to $t$).
	To do this we are going to consider the contribution from each smooth arc. 
	\begin{figure}[hpb]
		\centering
			\includegraphics[width=4in]{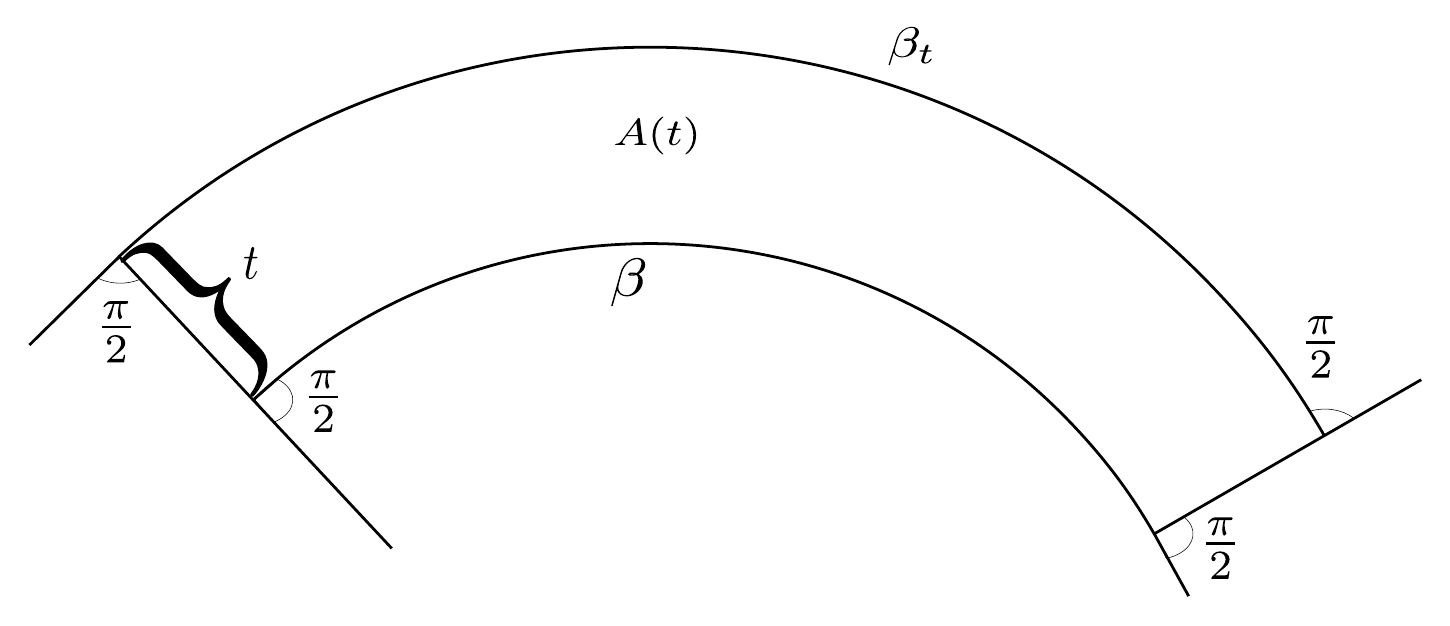}
		\caption{Smooth arcs of $\gamma$ and $\gamma_t$ }
		\label{hypsetup}
	\end{figure}
	
	Let $\beta$ be a smooth arc of the curve $\gamma$, and let $\beta_{t}$ be its corresponding arc in $\gamma_{t}$. Let $A(t)$ 
	be the area of the region between the two arcs, as in Figure \ref{hypsetup}. From the Gauss-Bonnet theorem we have that 
	\[ 2\pi + \int_{\beta_{t}}{k_{t}} - \int_{\beta}{k}\ = 2\pi + A(t) \] and
	\[  C(\beta_{t}) - C(\beta) = A(t). \]
	
	By looking at an infinitesimal change in $A$, we can see that $ A'(t) = \overline{L}(t)$, and by differentiating both sides we obtain
	\[
		C'(\beta_{t})= A'(t) = \overline{L}(t).
	\]
	
	To show that $L_{\gamma}'(t) = C_{\gamma}(t)$ we are going to apply a local Steiner formula to a polygonal approximation of the curve $\gamma$ by geodesic arcs $\beta_t$ and equidistant arcs $\beta_{t+dt}$ as in Figure \ref{hypsteiner}.
	
	\begin{figure}[htpb]
		\centering
			\includegraphics[width=4in]{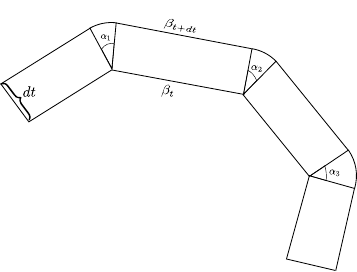}
			\caption{Geodesic approximation to $\gamma$}
			\label{hypsteiner}
	\end{figure}
  
  Since the circumference of a circle of radius $dt$ in the hyperbolic plane is $2\pi\sinh{dt}$ we see that
  
  \[\overline{L}(t+dt)= \overline{L}(t) + (\sum \alpha_{j})(\sinh{dt}).\]
	We also have that
  \[ \sinh{dt} = dt + \frac{dt^{3}}{3!}+ \dotsm .\]
  Ignoring the terms in $dt$ of higher order we see that 
  \[ \frac{\overline{L}(t+dt) - \overline{L}(t)}{dt} = \overline{L'}(t) = (\sum \alpha_{j}).\]
  
  Taking the limit as the number of vertices of $\beta_{t}$ goes to infinity we get that $(\sum \alpha_{j}) \rightarrow \overline{C}(t)$. 
  Therefore $\overline{L'}(t) = \overline{C}(t)$.

  Now we want to find the explicit formula for $\overline{C}(t)$ and $\overline{L}(t)$. To do this we are going to solve the system of differential equations
  \begin{gather*}
	  \overline{L'}(t) = \overline{C}(t) \\ 
	  \overline{C'}(t) = \overline{L}(t) 
  \end{gather*}
  from which we get that 
  \begin{gather*}
	  \overline{L}(t) = L_0\cosh{t} + C_0\sinh{t}\\
	 \overline{C}(t) = L_0\sinh{t} + C_0\cosh{t}
  \end{gather*}
  where $L_0= \overline{L}(0)$ and $C_0 = \overline{C}(0)$. We now note that the quantity $\overline{L}(t)^2 + 4\pi^2 - \overline{C}(t)^2$ is independent of the value of $t$. The following lemma shows that for a very large value of $t$, $\gamma_{t}$ is simple and smooth, so that we can apply the classical isoperimetric inequality. 
\input{hyperbolicsupport}
Let $t'$ be the point at which $\gamma_{t'}$ is a simple, smooth curve. By the classical isoperimetric inequality, we have that 
  \[\overline{L}(t')^2 - 4\pi A(\gamma_{t'}) - A(\gamma_{t'})^2 \geq 0.\]
  Using the Gauss-Bonnet theorem and the fact that $\gamma_{t'}$ is simple we get that
  \[ \overline{L}(t')^2 + 4\pi ^2 - \overline{C}(t')^2 \geq 0.\]
  But since the quantity on the left is independent of $t$ we see that
  \[L(\gamma)^2 + 4\pi^2 - C(\gamma)^2 \geq 0.\]
  
  If $\gamma$ is a geodesic circle, then the classical isoperimetric inequality applies, and the inequality in the statement of the theorem becomes an equality. Conversely, if we have equality for $\gamma$ then some equidistant of $\gamma$ is a geodesic circle, which implies that $\gamma$ itself is a geodesic circle. 
	\end{proof}

%% file: hyperbolicsupport.tex
\begin{lemma} Let $\gamma$ be a closed, horocyclically convex wave front with turning number one and let $\gamma _{t}$ be the family of equidistant fronts of $\gamma$. Then, for sufficiently large $t$, $\gamma_{t}$ is smooth and simple.  
\end{lemma}
\begin{proof}
We use the hyperbolic support function of Leichtwei\ss, found in \cite{hypsupport}. The hyperbolic support function is a periodic function, defined for each wave front. This function characterizes the curve, and is unique up to the choice of origin and the original direction. Note that this is a hyperbolic generalization of the support function that exists for euclidean curves (see \cite{support}, for example). We have that the hyperbolic support function is defined only for curves whose geodesic curvature is greater than or equal to one (are horocyclically convex), with turning number one. Adding a constant to the support function generates an equidistant curve to the original. 
The magnitude of the curvature $k$ is given by 
\begin{equation*}
|k| = \left| \frac{\ddot{H}\sinh H +(1+\dot{H}^2)\cosh H}{\ddot{H}\cosh H +(1+\dot{H}^2)\sinh H} \right|
\end{equation*}
where $H$ is the hyperbolic support function. The cusps correspond to the value of the denominator being zero. The claim is that if we add a sufficiently large constant to $H$ then the the geodesic curvature will be always greater than zero and less then $\infty$. Our result then follows, because under those conditions we have a smooth curve with everywhere positive curvature and turning number one, which implies that the curve is simple. 

First we show that for a large $t$ the denominator is not zero.
\begin{equation*}
\ddot{H}\cosh H +(1+\dot{H}^2)\sinh H = \frac{1}{2}(e^{H}(\ddot{H} + 1 + \dot{H}^2) + e^{-H}(\ddot{H} - 1 - \dot{H}^2))
\end{equation*}
Let 
\[a = \ddot{H} + 1 + \dot{H}^2 \qquad b = \ddot{H} - 1 - \dot{H}^2 . \]

The quantities $a$ and $b$ do not change when a constant is added to the function. Clearly both $a$ and $b$ can not be zero.
The only case where the denominator could be zero is if both $a$ and $b$ are non zero and have opposite signs. In this case if
we add a large constant to $H$, the $e^{H}$ term dominates and the denominator will not be zero. The same argument works for the 
numerator of the fraction. Since $H$ is periodic there exists a constant such that both the numerator and the denominator
of the expression of the geodesic curvature never vanish.

\end{proof}

%% file: sphericaldiffeq.tex
\begin{theorem}[$S^2$]\label{thm:sphdiffeq}
	Let $T_{l}(\gamma)$ be the function that sends a rear wheel curve $\gamma$ to the corresponding front wheel curve for a bicycle of length $l$. The condition $T_{l}(\gamma)=\Gamma$ is equivalent to the differential
	equation on the function $\alpha(s)$:

	\begin{equation}\label{sphdiffeq}
	\frac{d\alpha(s)}{ds}+\kappa(s)=\cot(l)\sin(\alpha).\end{equation}
	As well,
	\begin{equation}\label{sphdifflength}\left| \frac{dt}{ds} \right| = \left| \cos(\alpha) \right|.\end{equation}
\end{theorem}
\begin{proof}
	We will denote by $v$ the tangent vector to $BA$ at $\Gamma(s)$ and by $\tilde{v}$ the tangent vector to $BA$ at $\gamma(s)$, as in Figure \ref{fig:notation}. 
	\begin{figure}[ht]
		\centering
			\includegraphics[scale=1]{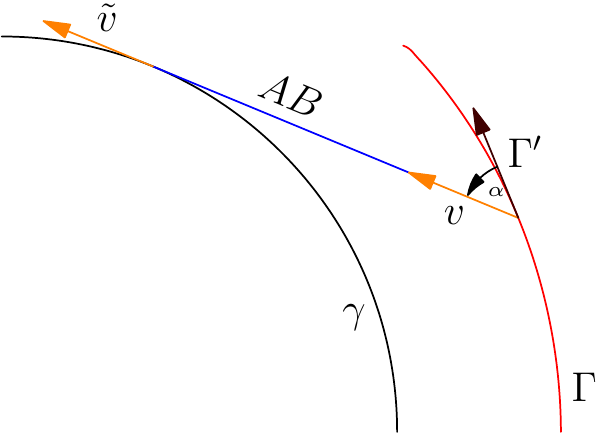}
		\caption{Notation of Theorem \ref{thm:sphdiffeq}}
		\label{fig:notation}
	\end{figure}
	First we write down an expression for $v$ and use it to obtain an expression for $\gamma$:
	\[v=\sin(\alpha)(\Gamma(s)\times\Gamma'(s))+\cos(\alpha)\Gamma'(s).\]
	\[\gamma(s)=\cos(l)\Gamma(s)+\sin(l)v\]
	Note that $\Gamma''(s)=-\Gamma(s)+\kappa(\Gamma(s)\times\Gamma'(s))$.
	We can use this to find the derivative $\gamma'=\dfrac{d\gamma}{ds}$:
	\begin{eqnarray*}
	\gamma'(s) & = &(\Gamma(s)\times\Gamma'(s)) (\sin(l)\cos(\alpha)(\kappa+\alpha')) +\\
	 &  &(\Gamma'(s))(\cos(l)-\sin(l)\sin(\alpha)(\kappa+\alpha'))+\\
	 &  &(\Gamma(s))(-\sin(l)\cos(\alpha)).\end{eqnarray*}
	We can also obtain an expression for $\tilde{v}$:

	\[\tilde{v}=\cos(l)v-\sin(l)\Gamma(s).\]
	Now, $\tilde{v}$ and $\gamma'(s)$ are parallel, so their cross product must be 0:
	\[0=\tilde{v}\times\gamma'(s)=A(\Gamma(s))+B(\Gamma'(s))+C(\Gamma(s)\times\Gamma'(s))\]
	Since these three vectors are orthogonal and non-zero, all of A, B,
	and C must be equal to 0. This gives us that either
	\[\frac{d\alpha}{ds}+\kappa=\cot(l)\sin(\alpha)\]
	or
	\[\sin(l)\cos(\alpha) = 0 \qquad \cos(l) = 0 \qquad \sin(\alpha) = 0. \]
	This system cannot be satisfied for any pair of $\alpha$ and
	$l$, however, so the 1st claim of the theorem is proven. 

	To establish the 2nd claim, we write the expression for $\gamma'(s)$ as above and take its
	magnitude. Using the differential equation that we have just derived, we substitute and simplify, which yields the result. 
\end{proof}

\begin{corollary}
	Parameterizing $\alpha$ by $t$, we have that 
	\begin{equation}\label{smallcurv}k=\frac{\tan(\alpha(t))}{\sin(l)}.\end{equation}
\end{corollary}
\begin{proof}
	We parameterize everything below by $t$. Let $\sigma$ denote the orientation of the rear wheel ($+1$ if the
	same direction as $\gamma'$ and $-1$ otherwise). We have that
	\begin{align*}
		\Gamma&=\cos(l)\gamma+\sigma\sin(l)\gamma' \\
		\Gamma'&=-\sigma\sin(l)(\gamma)+\cos(l)(\gamma')+\sigma\sin(l)k(\gamma\times\gamma') \\
		v&=\sin(l)\gamma-\sigma\cos(l)\gamma'.
	\end{align*}
	And so 
	\[ \Gamma'\times v=(\cos(l)\sin(l)k)\gamma+(\sigma\sin^{2}(l)k)\gamma'\]
	We can also write
	\[
	\Gamma'\times v=(\sin(\alpha)||v||||\Gamma'||)\Gamma=\left(\sin(\alpha)\sqrt{1+\sin^2(l)k^2}\right)\Gamma\]
	Plugging in our formula for $\Gamma$ from above and equating the
	two expressions we find that
	\[\sin(l)k=\sin(\alpha)\sqrt{1+\sin^2(l)k^2}\]
	Substituting we see this is solved for 
	\[\sin(l)k=\tan(\alpha).\]
\end{proof}

%% file: hyperbolicdiffeq.tex
\begin{theorem}[$H^2$]
The condition $T_{l}(\gamma)=\Gamma$ is equivalent to the differential
equation on the function $\alpha(s)$:

\begin{equation}\label{hypdiffeq}
\frac{d\alpha(s)}{ds}+\kappa(s)=\coth(l)\sin(\alpha).\end{equation}
As well,
\[ \left| \frac{dt}{ds} \right| = \left| \cos(\alpha) \right|.\]
\end{theorem}
\begin{proof}
	First we write down an expression for $v$ and use it to obtain an expression for $\gamma$:
	\[ v = \cos(\alpha) \Gamma' + \sin(\alpha)\Gamma'_{\perp}.\] 
	We have that $v$ is a unit tangent vector in the direction of the geodesic connecting $\gamma$ and $\Gamma$, so we can use the standard parameterization of geodesics in the hyperboloid model to obtain an expression for $\gamma$:
	\[ \gamma = \cosh(l)\Gamma + \sinh(l) v .\]
	By the hyperbolic Frenet-Serret formulas, we have that 
	\[ \Gamma'' = \Gamma + \kappa \Gamma'_\perp \quad \text{and} \quad (\Gamma'_\perp)' = -\kappa \Gamma'.\]
	Using this, we can find an expression for $\gamma' = \dfrac{d\gamma}{ds}:$
	\begin{eqnarray*}
		\gamma' & = &\Gamma(\sinh(l)\cos(\alpha)) \\
		& & \Gamma'(\cosh(l) - \sinh(l)\sin(\alpha)(\frac{d\alpha}{ds} + \kappa))\\
		& & \Gamma'_\perp (\sinh(l)\cos(\alpha)(\frac{d\alpha}{ds} + \kappa)).
	\end{eqnarray*}
	As in the spherical case, we can find $\tilde{v}$ by transporting $v$ along the geodesic $AB$:
	\[ \tilde{v} = \sinh(l)\Gamma + \cosh(l) v .\]
	We must have that $\tilde{v}$ and $\gamma'$ must be parallel (in the Euclidean sense), and so we can equate their Euclidean cross product to 0, giving us
	\[ 0 = \gamma' \wedge \tilde{v} = A(\Gamma' \wedge \Gamma) + B(\Gamma'_\perp \wedge \Gamma) + C(\Gamma' \wedge \Gamma'_\perp) .\]
	We have that $\Gamma, \Gamma'$ and $\Gamma'_\perp$ are linearly independent, so $A,B,C$ must all be 0. This gives us that either
	\[ \frac{d\alpha}{ds} + \kappa = \coth(l) \sin(\alpha)\]
	or 
	\[\sinh(l) = 0 \qquad \cos(\alpha)\sinh(l) =0 \qquad \cosh(l) = 0. \]
	This system cannot be satisfied for any pair of $\alpha$ and
	$l$, however, so we are done.
\end{proof}

\begin{corollary}[$H^2$]\label{cor:hypercurvature}
	Parameterizing $\alpha$ by $t$, we have that 
	\begin{equation}\label{hypercurvature}k=\frac{\tan(\alpha(t))}{\sinh(l)}.\end{equation}
\end{corollary}

\begin{proof}
We parameterize everything below by $t$. Let $\sigma$ denote the orientation of the rear wheel ($+1$ if the
same direction as $\gamma'$and $-1$ otherwise). Note
\[
\Gamma=\cosh(l)\gamma+\sigma\sinh(l)\gamma'\]
Now, by $\wedge$ denote the Lorentz cross product given by 
\[
x\wedge y=\left|\begin{array}{ccc}
i & j & -k\\
x_{1} & x_{2} & x_{3}\\
y_{1} & y_{2} & y_{3}\end{array}\right|.\]
After substituting $\gamma''=\gamma+k(\gamma\wedge\gamma')$ we get
\[
\Gamma'=\sigma\sinh(l)(\gamma)+\cosh(l)(\gamma')+\sigma\sinh(l)k(\gamma\wedge\gamma').\]
Now, we can find $v$ by parameterizing the arc of the bicycle starting at $\gamma$, taking the derivative at
$l$, and flipping it to get:
\[
v=-(\sinh(l)\gamma+\sigma\cosh(l)\gamma')=-\sinh(l)\gamma-\sigma\cosh(l)\gamma'.\]
Then \[
\Gamma'\wedge v=-\sinh^{2}(l)k(\gamma\wedge\gamma')\wedge\gamma-\cosh(l)\sinh(l)k(\gamma\wedge\gamma')\wedge\gamma'.\]
However, it can be verified that $(\gamma\wedge\gamma')\wedge\gamma=\gamma'$
and $(\gamma\wedge\gamma')\wedge\gamma'=\gamma$, so
\[
\Gamma'\wedge v=-\sinh^{2}(l)k\gamma'-\cosh(l)\sinh(l)k\gamma.\]
We also have that $\Gamma'\wedge v=\sin(\alpha)(||\Gamma'||_M)(||v||_M)(\Gamma),$
where $||\cdot||_M$ is the Minkowski norm. This gives us:
\[
\Gamma'\wedge v=\sin(\alpha)\sqrt{1+\sinh^{2}(l)k^{2}}(\Gamma).\]
Substituting in for $-\Gamma$ and equating our two expressions we find 
\[
\sin(\alpha)\sqrt{1+\sinh^{2}(l)k^{2}}=\sinh(l)k.\]
Which gives us that $\sinh(l)k=\tan(\alpha)$.
\end{proof}

%% file: monodromy.tex
Having established differential equations that describe the bicycle motion on the sphere and in the hyperbolic plane, we are now in a position to talk about the monodromy map associated to a front wheel curve $\Gamma$. 

Note that in the euclidean plane we have the ability to identify circles centered at different points of $\Gamma$ via parallel translation. On the sphere or in the hyperbolic plane, this is not possible. Instead, we identify circles at different points of $\Gamma$ so that the velocity vector $\Gamma '$ makes an angle of 0 degrees. With this convention, we can define the monodromy map as above, and we see that the monodromy map is always a M\"obius transformation, but let us first make sense of that statement. 
\begin{definition}
	A \emph{M\"obius transformation} is a fractional linear map $M: \mathbb{C} \rightarrow \mathbb{C}$, of the form
	$$M(z) = \frac{az + b}{cz + d}$$ for some $a,b,c,d \in \mathbb{C}$. In matrix form:
	\[ M = \begin{bmatrix}
		a & b \\
		c & d
	\end{bmatrix} .
	\]
\end{definition} 
	 We have that the orientation-preserving isometries of $H^2$ in the upper half-plane model are those M\"obius transformations with real coefficients and such that $ad-bc = 1$. These correspond to the group $SL(2,\mathbb{R})$ of $2 \times 2$ real matrices with determinant 1. The isometries extend to the real line and the point at infinity, which can be identified with $\mathbb{RP}^1$ or $S^1$. This gives us a way for M\"obius transformations to act on $S^1$, which is the domain and range of our monodromy map. 
\begin{theorem}\label{theorem:monodromy}
	Let $\Gamma$ be a front wheel bicycle curve. Then the monodromy map $M$ associated to $\Gamma$ is a M\"obius transformation. 
\end{theorem}
\begin{proof}
	First move from the coordinate $\alpha$ to the projective coordinate $y=\tan \frac{\alpha}{2}$. We can then rewrite the differential equation \eqref{sphdiffeq} as \[y'= -\dfrac{\kappa}{2}y^{2}  + \cot(l) y - \dfrac{\kappa}{2}.\]
	We then have that, treating $y'$ as a vector field and with $\dfrac{d}{dy}$ being a unit vector on $\mathbb{RP}^1$,
	\[y'(s)= -\dfrac{\kappa}{2}(s)y^{2}\dfrac{d}{dy}(y) + \cot(l) y\dfrac{d}{dy}(y) - \dfrac{\kappa}{2}(s)\dfrac{d}{dy}(y),\]
	so that the vector field $y'$ is a combination of the vector fields $y^2\dfrac{d}{dy}$, $y\dfrac{d}{dy}$, and $\dfrac{d}{dy}$ with $s$-dependent coefficients. These vector fields generate the Lie algebra of $SL(2)$, which acts on $\mathbb{RP}^1$ via M\"obius transformations. So $y$ is a transformation whose infinitesimal action is the same of that of a M\"obius transformation, implying that the monodromy is in fact a M\"obius transformation.   
\end{proof}
The M\"obius transformations that we are dealing with are identified with isometries of the hyperbolic plane, which come in 3 types: elliptic, parabolic, and hyperbolic. These transformations have 0, 1, and 2 fixed points on the circle at infinity, respectively. Note that if the front wheel curve $\Gamma$ is closed, a fixed point of its associated monodromy map corresponds to a \emph{closed} rear wheel curve.
\begin{theorem}[Both]\label{smallhyperbolic}
	For sufficiently small values of $l$, the monodromy map is hyperbolic. 
\end{theorem}
\begin{proof}
	Take the limit as $l \rightarrow 0$ in \eqref{sphdiffeq} (for $S^2$) or \eqref{hypdiffeq} (for $H^2$). We obtain $\sin(\alpha) = 0$, which has two solutions $\alpha = 0, \pi$. For small enough values of $l$, these solutions will survive, and so the monodromy map will have two fixed points corresponding to them.
\end{proof}
\begin{theorem}[Both]
	Let $M$ have a fixed point $\theta_0$ and let $\gamma$ be the closed curve of the rear wheel corresponding to $\theta_0$. Then we have 
	\begin{equation}\label{derivlength} M'(\theta_0) = e^{\text{-Length}(\gamma)}. \end{equation}
\end{theorem}
\begin{proof}
	We consider the spherical case first. If $M$ has a fixed point, we have that there exists an $\alpha(x)$ that is an $L-$periodic solution to \eqref{sphdiffeq}. Consider an infinitesimal perturbation of this solution, $\alpha(x) + \epsilon \beta(x)$. First, we have that $M'(\theta_0) = \dfrac{\beta(L)}{\beta(0)}$. For $\alpha + \epsilon \beta$ to satisfy \eqref{sphdiffeq}, we must have 
	\begin{align*}
		(\alpha + \epsilon \beta)' + \kappa &= \cot(l)\sin(\alpha + \epsilon \beta) \\
		\alpha' + \epsilon \beta' + \kappa &= \cot(l)(\sin \alpha \cos \epsilon \beta + \cos \alpha \sin \epsilon \beta). \\
	\end{align*}
	Using the fact that $\alpha$ is a solution to \eqref{sphdiffeq} and that $\epsilon$ is small, we get
	\begin{align*}
		\cot(l)\sin(\alpha) + \epsilon \beta' &= \cot(l)\sin\alpha + \epsilon \beta \cos\alpha \\
		\beta' - \beta \cos\alpha &= 0.
	\end{align*}
	We then have that
	\[M'(\theta_0) = \frac{\beta(L)}{\beta(0)} = e^{\int_0^L \cos(\alpha(x)) dx} = e^{\text{-Length}(\gamma)} .\]
	In the hyperbolic plane, we see that (after comparison of \eqref{sphdiffeq} to \eqref{hypdiffeq}) the above calculations go through with only minor modifications that do not affect the final result. 
\end{proof}

\begin{corollary}[Both]\label{paraboliclength}
	$M$ is parabolic if and only if the algebraic length of $\gamma$ is 0.
\end{corollary}
\begin{proof}
	In the parabolic case, $M'(\theta_0) = 1$, which implies that $L(\gamma) = 0.$ Conversely, if $L(\gamma) =0$, then $M'(\theta_0) = 1$, which implies that $M$ is parabolic, as the derivatives at fixed points are reciprocal to each other. 
\end{proof}
\begin{corollary}[Both]
	In the parabolic case, $\gamma$ has cusps. 
\end{corollary}
\begin{proof}
	If a closed curve is to have 0 length, there must be arcs with opposite parity. This implies that such a curve must have cusps, as this is the only way to get arcs with different parities. 
\end{proof}

\begin{remark}[$S^2$]
	On the sphere, we have the notion of a derivative curve as discussed in \cite{arnoldsphere}. Given a spherical wave front $\gamma$ we associate to it the curve $\Gamma$ obtained by moving every point a distance of $\frac{\pi}{2}$ in the direction tangent to the curve at that point. This is the same as using $\gamma$ as a back wheel curve for a bicycle of length $\frac{\pi}{2}$. By a result in \cite{arnoldsphere}, we have that the derivative curves are the same for any equidistant front of $\gamma$, and so the $\Gamma$ produced as the front wheel curve has a whole family of back wheel trajectories that generate it. This implies that the monodromy map associated to $\Gamma$ is the identity, as it has an infinite number of fixed points (corresponding to the family of equidistance fronts of $\gamma$ that generate it). 
\end{remark}

%% file: menzin.tex
In \cite{tabachnew}, Levi and Tabachnikov prove an old conjecture of Menzin for bicycle curves in the plane. The conjecture states that if a closed, convex front wheel curve bounds area greater than $\pi$ (the area of a unit circle) then the associated monodromy map is hyperbolic (for $l = 1$, other choices of $l$ simply scale the value $\pi$). On the sphere, a circle of radius $l$ has area $2 \pi (1 - \cos l )$, so it seems plausible that the monodromy map would be hyperbolic in the case that a closed, convex front wheel curve bounds area greater than this. This conjecture is in fact true, as the following shows.
\begin{theorem}[$S^2$]\label{theorem:menzin}
	Let $\Gamma$ be a closed, convex curve, oriented properly. Then if $A(\Gamma)$ is greater than $2 \pi(1 - \cos l)$, the corresponding monodromy map is hyperbolic. 
\end{theorem}
\begin{proof}
	We argue by contradiction, assuming that $A(\Gamma) > 2\pi(1-\cos l)$ and $M(\Gamma)$ is not hyperbolic. We need to show that 
	\begin{align*}
		A(\Gamma) &\leq 2\pi(1 - \cos l) \\
		2\pi - \acc(\Gamma) &\leq 2\pi(1-\cos l) \\
		\acc(\Gamma) &\geq 2 \pi \cos l.
	\end{align*}
	Let the length of the bicycle vary as $tl$ (with $t$ from 0 to 1). For small $t$, we have that $M$ is hyperbolic, by Theorem \ref{smallhyperbolic}. This cannot continue for all $t$, however, so there must be a $t_0 l = l' \leq l$ such that $M$ is parabolic. Let $\gamma$ be the closed back wheel curve corresponding to the fixed point of $M$ for $l'$. Because $l' \leq l$ we have that $\cos l' \geq \cos l$, so to establish the above inequality it suffices to show that 
	\[\acc(\Gamma) \geq 2 \pi \cos l'.\]
	We now prove a lemma relating the curvature of the back and front wheels that will allow us to simplify this inequality.
	
	\begin{lemma}\label{lem:curvaturerelation} Letting $\kappa$ be the curvature of the front wheel and $k$ the curvature of the back wheel, we have
		\begin{equation}\label{curvaturerelation} \int\kappa(s)ds=\cos(l)\int k(t)dt, \end{equation}
		where integrals are taken over their respective curves and the elements of arclength are signed. 
		Equivalently,
		\begin{equation}
			\acc(\Gamma) = \cos(l)\acc(\gamma).
		\end{equation}
	\end{lemma}
	\begin{proof}
		We first rewrite the integral of the curvature of $\gamma$ using \eqref{smallcurv} \begin{eqnarray*} 
		\int k(t)dt & = & \csc(l)\int\tan(\alpha(t))dt\\
		 & = & \csc(l)\int\tan(\alpha(s))\frac{dt}{ds}ds\\
		 & = & \csc(l)\int\tan(\alpha(s))\cos(\alpha(s))ds\\
		 & = & \csc(l)\int\sin(\alpha(s))ds.\end{eqnarray*}

		Note that $\cos(\alpha(s))$ is signed to reflect changes when passing through cusps. Integrating both sides of \eqref{sphdiffeq}, we get
		\[ \int\left(\frac{d\alpha(s)}{ds}+\kappa(s)\right)ds =  \int \cot(l)\sin(\alpha(s))ds .\]
		As $\alpha$ is periodic, the first term on the left hand side is 0, so combining this with the above we have 
		\begin{eqnarray*}
			\int\kappa(s)ds & = & \cot(l) \int\sin(\alpha(s))ds\\
			 & = & \cos(l)\int k(t) dt.
		\end{eqnarray*}
	\end{proof}
	\begin{remark}\label{hypcurvrel}
		If instead of using \eqref{smallcurv} and \eqref{sphdiffeq}, we use their hyperbolic counterparts \eqref{hypercurvature} and \eqref{hypdiffeq}, we obtain the following for $H^2$:
			\begin{equation*}\int\kappa(s)ds=\cosh(l)\int k(t)dt.\end{equation*}
	\end{remark}	
	In light of \eqref{curvaturerelation} and what is discussed in \cite{arnoldsphere}, the inequality that we seek to show is 
	\[ \int k = \acc(\gamma) \geq 2 \pi .\]
	To complete the proof, we need the following lemma.
	\begin{lemma}\label{smoothdual}
		The rear wheel curve $\gamma$ has no inflection points and an even number of double tangent points.
	\end{lemma}
	\begin{proof}
		First we show that $\gamma$ has an even number of double tangent points. 
		Consider the family of dual curves $\gamma^*_{\ell}$, as defined in \ref{def:dual}. When $\ell$ is small we know that $\gamma_{\ell}$ is convex, and therefore $\gamma^*_{\ell}$ is smooth and has no inflections. As $\ell$ varies $\gamma^*$ changes by a regular homotopy, since $\gamma_{\ell}$ has no inflection points for any values of $\ell$. Then $\gamma^*_{\ell}$ is regularly homotopic to a circle traversed once for all $\ell$ and thus has an even number of intersections. Therefore $\gamma_{\ell}$ has an even number of double-tangent points.
		
		Now we show that $\gamma$ has no inflection points. First, we compute an expression for $\kappa$ in terms of $k$. Using \eqref{smallcurv}, we get
		\[ k = \frac{\tan(\alpha)}{\sin l} \quad \rightarrow \quad \alpha = \arctan(k \sin l). \]
		Placing this expression into \eqref{sphdiffeq}, we obtain
		\begin{align*}
			\frac{d \arctan(k \sin l)}{ds} + \kappa &= \cot l \sin( \arctan (k \sin l)) \\
			\frac{k'}{1 + (k \sin l)^2} + \kappa &= \cot l  \frac{k \sin l}{\sqrt{1 + (k \sin l )^2}}.
		\end{align*}
		Suppose that there is an inflection point of $\gamma$. Consider the family of fronts $\gamma_l$, for varying $l$, and let $l'$ be the least value of $l$ for which $k =0$. For $l$ slightly greater than $l'$, there is a dimple in $\gamma$ that causes $k$ to change sign twice, starting from a positive value. Then there exists an $s$ such that $k'(s) = 0$ and $k(s) < 0$, and by the above equation we have that $\kappa < 0$. However, because $\Gamma$ is convex, we must have that $\kappa \geq 0$, and so this is impossible. 
	\end{proof}
	In light of the above lemma, we can use Theorem \ref{thm:wcurve}, which tells us that
	\[ \acc(\gamma)^{2} + L(\gamma)^{2} \geq 4 \pi^{2}.\]
	By Corollary \ref{paraboliclength}, we also have that $L(\gamma) = 0$, as $M(\Gamma)$ is parabolic. Putting it all together, we get that 
	\[ 
	\acc(\gamma) \geq 2 \pi,
	\]
	which gives us the desired contradiction.
\end{proof}

%% file: hyperbolicmenzin.tex
We now prove a version of the previous theorem in the hyperbolic plane. Because of some peculiarities of the hyperbolic plane, the hyperbolic Menzin's Conjecture is proven only for horocyclically convex curves (i.e. hyperbolic curves with curvature $\geq 1$).
	\begin{theorem}
		Let $\Gamma$ be a closed, horocyclically convex curve in the hyperbolic plane. Then if $A(\Gamma)$ is greater then $(2\pi \cosh{l} - 1)$,
		the corresponding monodromy map is hyperbolic. 
	\end{theorem}

	\begin{proof}
		We argue by contradiction, assuming that $A(\Gamma) > 2\pi(\cosh l - 1)$ and $M(\Gamma)$ is not hyperbolic. We need to show that 
					\begin{align*}
						A(\Gamma) &\leq 2\pi(\cosh l - 1) \\
						C(\Gamma) - 2\pi &\leq 2\pi(\cosh l - 1) \\
						C(\Gamma) &\leq 2 \pi \cosh l.
					\end{align*}
		Let the length of the bicycle vary as $t l$. For small $t$, we have that $M$ is hyperbolic. There exists a $t_0 l = l' \leq l$ such that $M$ is parabolic. Let $\gamma$ be the closed back wheel curve corresponding to the fixed point of $M$ for $l'$. Because $l' \leq l$ we have that $\cosh l' \leq \cosh l$, so to establish the above inequality it suffices to show that 
		\[C(\Gamma) \leq 2 \pi \cosh l'.\]
		Since 
		\[C(\Gamma) = \cosh(l')C(\gamma)\] 
		we need to show
		\[C(\gamma) \leq 2 \pi .\] 
		By virtue of the generalized isoperimetric inequality proven above we have  
		\[ L(\gamma)^{2} + 4\pi^2 - C(\gamma)^2 \geq 0.\]
		We also have that $L(\gamma) = 0$, as $M(\Gamma)$ is parabolic. Putting it all together, we get that 
		\[ C(\gamma) \leq 2 \pi\]
		which gives us the desired contradiction.
\end{proof}